\documentclass[oneside]{amsart}
\usepackage{amsmath,amssymb,color,amsthm,graphicx,cite}
\usepackage[T1]{fontenc}
\usepackage[utf8]{inputenc}
\usepackage{color,bm}
\graphicspath{{../../img_pdf/}{../img_pdf/}{./img_pdf/}}
\usepackage{enumerate} 
\usepackage[backref]{hyperref}
\RequirePackage{luatex85}
\usepackage[all]{xy}
\newtheorem{theorem}{Theorem}[section]
\newtheorem{proposition}[theorem]{Proposition}
\newtheorem{lemma}[theorem]{Lemma}
\newtheorem{corollary}[theorem]{Corollary}

\newtheorem{remark}{Remark}

\newtheorem{example}{Example}

\theoremstyle{definition}
\newtheorem{definition}{Definition}

\def\F{\mathcal{F} }

\def\R{\mathbb{R} } 
 
\def\Z{\mathbb{Z} }

\def\R{\mathbb{R} }

\makeatletter
\@namedef{subjclassname@2020}{%
  \textup{2020} Mathematics Subject Classification}
\makeatother

\author{Yusuke Imoto}
\address{Institute for the Advanced Study of Human Biology, Kyoto University, Yoshida Konoe-cho, Sakyo-ku, Kyoto, 606-8501, Japan}
\email{imoto.yusuke.4e@kyoto-u.ac.jp}

\author{Tomoo Yokoyama}
\address{Department of Mathematics, Faculty of Science, Saitama University, Shimo-Okubo 255, Sakura-ku, Saitama, 338-8570, Japan}
\address{Department of Mathematics \& Statistics, McMaster University, 1280 Main Street West, Hamilton, Ontario, L8S 4L8, Canada}
\address{The Fields Institute for Research in Mathematical Sciences, 222 College Street, Toronto, Ontario, M5T 3J1, Canada}
\email{tyokoyama@rimath.saitama-u.ac.jp}

\usepackage{color,bm}

\title[Multi-parameter persistence for maximizing effects from inputs]{Multi-parameter persistence in dynamical systems for maximizing effects of control inputs}



\begin{document}

\begin{abstract}
We introduce a new topological method to naturally extend a partial function $h \colon X \rightharpoonup [-\infty, \infty]$ on a ``generalization'' of a metric space $X$ equipped with a dynamical system $f \colon X \rightharpoonup X$, to a function $h_f^{\varepsilon\text{-}\ell^p} \colon X \to [-\infty,\infty]$ with parameters $\varepsilon,p$, which allows us to apply existing topological data analysis techniques to functions defined on the entire space. Moreover, given a function $h$ that evaluates the ``quality'' of points within $\mathop{\mathrm{dom}}h$, using this extended function, one can construct a sufficient condition for the existence of an optimal $\varepsilon$-perturbation path from any point into $\mathop{\mathrm{dom}}h$ that minimizes the value of $h$ under the condition $X = \mathop{\mathrm{dom}} f  \sqcup \mathop{\mathrm{dom}}h = \bigsqcup_{n = 0}^\infty f^{-n}(\mathop{\mathrm{dom}}h)$. In addition, if the domain $X$ is finite, then the function $h_f^{\varepsilon\text{-}\ell^p} \colon X \to [-\infty,\infty]$ can be computed recursively. As an application, for a given partial evaluation function on a space equipped with a dynamical system, one can construct a three-parameter filtration associated with its extension, which naturally identifies minimal paths. This clarifies the relationship among three factors: the evaluation of the cost norm, the strength of control, and the resulting value.
\end{abstract}

\maketitle

\section{Introduction}

\subsection{Topological data analysis perspective}

Standard one-parameter persistent homology has developed rapidly in both theory and applications since the turn of the century. 
However, real-world data typically involve multiple parameters, such as scale, density, time, and other external or internal variables. 
Therefore, extending persistent homology from one parameter to multi-parameters has become a natural direction in topological data analysis. 
Although simple barcode representations are no longer available in the same way as in the one-parameter case \cite{carlsson2009theory}, computational and theoretical foundations for two-parameter persistent homology have been developed, including methods for computing minimal presentations and bigraded Betti numbers \cite{lesnick2022computing}. 
Moreover, multi-parameter persistence has begun to be applied to practical data analysis; for example, multi-parameter persistent homology landscapes have been used to identify immune-cell spatial patterns in tumors \cite{vipond2021multiparameter}. 
On the other hand, there remains a need for a framework that naturally constructs a multi-parameter filtration from a single partial evaluation function, thereby enabling the application of multi-parameter persistence to data equipped with both dynamics and scalar-valued evaluations.

\subsection{Optimal transport perspective}
The question of how to minimize costs while achieving better outcomes as in optimal transport problems is widely researched topic from both theoretical and applied perspectives. 
For instance, such researches include the pathfinding problem via combinatorial optimization, the shortest path problem in graph theory, and finding the shortest path on Riemannian manifolds using the calculus of variations.

Recently, a method has been proposed for constructing a ``cost function'' that determines the cost required to return to the initial point \cite{yokoyama2025coarse,yokoyama2025coarse_nonwandering}. 
Furthermore, another method has also been proposed to construct a function that evaluates the cost required to achieve a ``good'' outcome \cite{imoto2026filtration}. 
In numerous optimization problems, the quality of a result is evaluated not only by binary functions, such as success or failure, but also by scalar-valued functions that quantify the degree of desirability of an outcome. 
For example, numerical studies on cloud seeding evaluate weather interventions through scalar-valued quantities such as accumulated precipitation and the localization of heavy rainfall \cite{hiraga2026numerical}. 
Similarly, in aerodynamic shape optimization, scalar-valued performance measures such as the lift-to-drag ratio, lift coefficient, and drag coefficient are optimized to improve airfoil performance \cite{dussauge2023reinforcement}.
On the other hand, since the domain of an evaluation function does not generally span the whole space, a method for effectively extending this evaluation function to the whole space is also required in data analysis. 

\subsection{From the perspective of the extension problem of partial functions}
From a theoretical perspective, the problem of whether a partial function can be extended to a broader domain (e.g. the whole space) while preserving certain structures has been studied across various fields.
For instance, such results include the Hahn-Banach theorem in functional analysis, which extends a linear functional while bounding its norm; the Tietze extension theorem in topology, which provides a continuous extension; the identity theorem in complex analysis, which ensures the uniqueness of a holomorphic extension; and Szpilrajn's extension theorem in order theory, which extends a partial order while preserving existing comparability relations.

\subsection{Description of main results}
We introduce a framework to describe how to minimize costs and achieve better results, even in situations where the quality of the outcome is evaluated by a scalar-valued function. 
Furthermore, we demonstrate that for a finite space, an algorithm can be constructed to find the optimal path for any arbitrary point.
Moreover, for a given partial evaluation function on a space equipped with a dynamical system, we constructed a multi-parameter filtration associated with its extension, which naturally identifies minimal paths.

The present paper consists of four sections.
In the next section, we recall and introduce some concepts of dynamical systems. 
In addition, we state fundamental properties of new concepts and illustrate an example to understand new concepts. 
In \S~3, we introduce a controlled effect function to describe a ``minimizing'' controlled path.
Moreover, we demonstrate the existence of a ``minimizing'' controlled path under perturbations (Theorem~\ref{th:minimizing}). 
In the final section, we construct an algorithm to extend partial evaluation function (Theorem~\ref{th:01} and Theorem~\ref{th:02}) and a multi-parameter filtration associated with the extension (Theorem~\ref{th:03} and Theorem~\ref{th:04}).

\section{Preliminaries}

\begin{definition}\label{def:filtration}
Let $X$ be a set and $\F = \{ F_i \mid i \in I\} \subset 2^X$ a family indexed by a directed set $I$, where $2^X$ is the power set of $X$. The family $\F$ is a {\bf (multi-parameter) filtration} if $X = \bigcup_{i \in I} F_i$ and $F_{i_1} \subseteq F_{i_2}$ for any pair $i_1 \leq i_2 \in I$.
\end{definition}

\begin{definition}\label{def:partial_map}
A map $f \colon X' \to Y$ from a subset $X' \subseteq X$ to a set $Y$ is called a {\bf partial map} from $X$ to $Y$ and denoted by $f \colon X \rightharpoonup Y$.
\end{definition}

\begin{definition}\label{def:positive_orbit}
For any partial map $f \colon X \rightharpoonup X$ and any point $x \in X$, the {\bf non-negative orbit} $O^{\geq 0}_f(x)$ is defined as follows:
\[
O^{\geq 0}_f(x) =
\begin{cases}
\{ x \} & \text{if } x \notin \mathop{\mathrm{dom}}f \\
\{ x, f(x), \ldots , f^n(x) \} & \text{if there is }n \in \Z_{>0} \text{ such that }f^n(x) \notin \mathop{\mathrm{dom}}f \\
& \hspace{6.5pt} \text{ and }  \{ x, f(x), \ldots , f^{n-1}(x) \} \subseteq  \mathop{\mathrm{dom}}f \\
\{ f^n(x) \mid n \in \Z_{\geq 0} \}  & \text{otherwise}
\end{cases}
\]
\end{definition}


\begin{definition}\label{def:cost}
Let $X$ be a set and $c \colon X^2 \to [-\infty,\infty]$ a function.
The function $c$ is a {\bf cost function} if $c(x,x) = 0$ for any $x \in X$. 
\end{definition}

Note that we do not require transitivity in the previous definition to analyze various phenomena. 
Moreover, we allow the cost c to be negative to incorporate scenarios, such as when a decrease in height between two points provides potential energy, which can be viewed as a negative cost.


\begin{definition}\label{def:metric-like}
A cost function $c \colon X^2 \to [-\infty,\infty]$ is {\bf non-degenerate} if $c^{-1}(0) = \{ (x,x) \mid x \in X\}$. 
\end{definition}

\begin{definition}\label{def:metric-like_02}
A cost function $c \colon X^2 \to [-\infty,\infty]$ is {\bf non-negative-valued} if $\mathop{\mathrm{dom}}c \subseteq [0,\infty]$. 
\end{definition}

\subsection{Controlled paths}

Let $f \colon X \rightharpoonup X$ be a partial map on a set $X$ with a cost function $c \colon X^2 \to [-\infty,\infty]$ and a value $p \in [1,\infty]$. 
Recall that the sign function $\mathop{\mathrm{sgn}} \colon \R \to \{-1,0,1\}$ on $\R$ is defined by $\mathop{\mathrm{sgn}}^{-1}(-1) = \R_{<0}$, $\mathop{\mathrm{sgn}}^{-1}(0) = \{ 0 \}$, and $\mathop{\mathrm{sgn}}^{-1}(1) = \R_{>0}$. 

\begin{definition}\label{def:e-p-controlled-path}
For any $\varepsilon \in [-\infty,\infty]$ and any $n \in \Z_{> 0}$, a  sequence $(x;x_0, \ldots , x_{n-1};y)$ 
is an {\bf $\bm{\varepsilon}$-$\bm{\ell^p}$-controlled path} of length $n$ from the initial state $x \in X$ to the terminal state $y \in X$ if $(x_0, \ldots , x_{n-1}) \in (\mathop{\mathrm{dom}}f)^{n}$, $y = f(x_{n-1})$, and 
$\varepsilon \geq \| ( \varepsilon_i )_{i=0}^{n-1} \|_{p}$, where $\varepsilon_0 := c(x, x_{0})$ and $\varepsilon_i := c(f(x_{i-1}), x_{i})$ for any $i \in \{ 1, \ldots , n-1 \},$ and  
\[
\| ( \varepsilon_i )_{i=0}^{n-1} \|_{p} := 
\begin{cases}
\mathop{\mathrm{sgn}}\left( \sum_{i=0}^{n-1} \mathop{\mathrm{sgn}}(\varepsilon_i) \vert \varepsilon_i \vert^p
\right) \vert \sum_{i=0}^{n-1} \mathop{\mathrm{sgn}}(\varepsilon_i) \left| \varepsilon_i \vert^p
\right|^{1/p}
&  (p \in [1,\infty)) \\
\max\{ \varepsilon_0, \varepsilon_1, \ldots , \varepsilon_{n-1} \} &  (p = \infty) 
\end{cases}
\]
as shown in Figure~\ref{fig:chains}. 
\end{definition}

Note that, for any non-negative-valued cost function $c \colon X^2 \to [0,\infty]$ and any sequence $(x_0, \ldots , x_{n-1}) \in X^{n}$,  
\[
\| ( \varepsilon_i )_{i=0}^{n-1} \|_{p} = 
\begin{cases}
\left( \vert \varepsilon_0 \vert^p  + \vert \varepsilon_1 \vert^p + \cdots + \vert \varepsilon_{n-1}\vert^p \right)^{1/p} &  (p \in [1,\infty)) \\
\max\{ \varepsilon_0, \varepsilon_1, \ldots , \varepsilon_{n-1} \} &  (p = \infty) 
\end{cases}
\]
is the $\ell^p$-norm of $( \varepsilon_i )_{i=0}^{n-1}$, where $\varepsilon_0 = c(x, x_{0})$ and $\varepsilon_i = c(f(x_{i-1}), x_{i})$ for any $i \in \{ 0,1, \ldots , n-1 \}$. 
We introduce the following notations.

\begin{definition}
For any $\varepsilon \in [-\infty,\infty]$ and any point $x \in X$, denote by $[x]_f^{\varepsilon\text{-}\ell^p}$ the set of points to which  there is an $\varepsilon\text{-}\ell^p$-controlled path from $x$.
\end{definition}

\begin{definition}
For any $\varepsilon \in [-\infty,\infty]$, any $x \in X$, and any $\varepsilon$-$\ell^p$-controlled path $\gamma = (x;x_0, \ldots , x_{n-1};f(x_{n-1}))$, define the subset $N(\gamma) := \{x, f(x_0), f(x_1), \ldots , f(x_{n_1}) \}$. 
\end{definition}

Notice that, for any $\varepsilon$-$\ell^p$-controlled path $\gamma = (x;x_0, \ldots , x_{n-1};f(x_{n-1}))$,  a point $x'$ is contained in $N(\gamma)$ if and only if either $x' = x$ or the point $x'$ is the terminal state of an $\varepsilon$-$\ell^p$-controlled path $(x;x_0, \ldots , x_{k-1};f(x_{k-1}))$ for some natural number $k \in \{1, \ldots , n\}$.

\subsubsection{Existence of a controlled path into a subset}

We have the following relations.

\begin{definition}
Fix a number $\varepsilon \in [-\infty,\infty]$, a point $x \in X$, and a subset $A \subseteq X$. 
Define the relation $x \overset{\exists}{\underset{f,\varepsilon\text{-}\ell^p, 0}{\rightharpoonup}} A$ as follows:
\[
x \overset{\exists}{\underset{f,\varepsilon\text{-}\ell^p, 0}{\rightharpoonup}} A \text{ if } x \in A
\]
Moreover, for any $n \in \Z_{> 0}$, the relation $x \overset{\exists}{\underset{f,\varepsilon\text{-}\ell^p, n}{\rightharpoonup}} A$ holds if there is an $\varepsilon$-$\ell^p$-controlled path of length $n$ from the initial state $x$ to the terminal state in $A$. 
 
Moreover, the relation $x \overset{\exists}{\underset{f,\varepsilon\text{-}\ell^p}{\rightharpoonup}} A$ holds if there is a non-negative number $n \in \Z_{\geq 0}$ such that $x \overset{\exists}{\underset{f,\varepsilon\text{-}\ell^p, n}{\rightharpoonup}} A$. 
\end{definition}

\begin{remark}
For any non-negative valued cost function, any $\varepsilon$-$\ell^\infty$-controlled path corresponds with the $\varepsilon$-controlled path in \cite{imoto2026filtration}, and any $\varepsilon$-$\ell^1$-controlled path corresponds with the $\varepsilon_{\Sigma}$-controlled path in \cite{imoto2026filtration}.
\end{remark}

\begin{figure}[t]
\begin{center}
\includegraphics[scale=0.75]{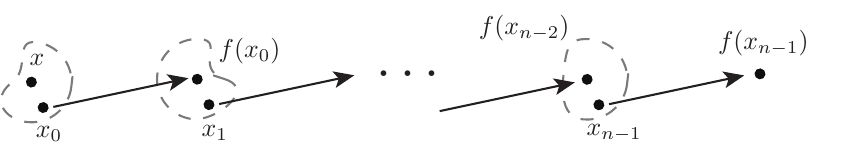}
\end{center} 
\caption{An $\varepsilon$-$\ell^p$-controlled path of length $n$ from the initial state $x$ to the terminal state $f(x_{n-1})$.}
\label{fig:chains}
\end{figure}

\begin{figure*}[t]
\begin{center}
\makebox[\textwidth][c]{%
\includegraphics[width=1.\textwidth]{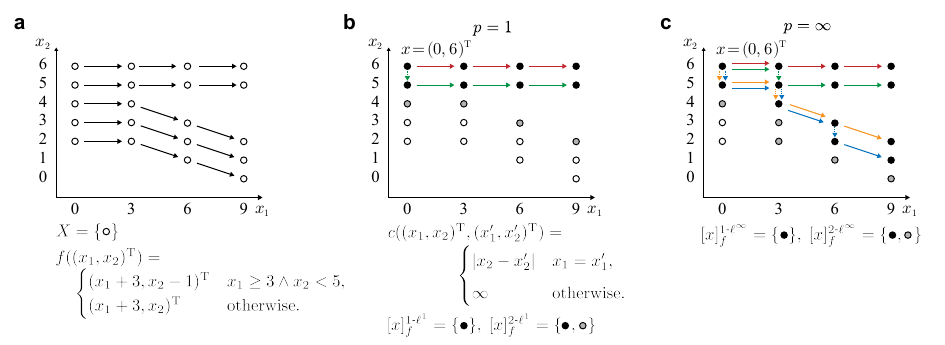}
}
\end{center} 
\caption{
Schematic illustration of controlled paths.
\textbf{a}, A setting of the dynamical system $f$ and the set $X$, where $X$ consists of the open circles. Black arrows represent the map $f$. The dynamics change uniformly on the left side of the state space and then branch into two trajectories after an intermediate region.
\textbf{b}, \textbf{c}, Examples of controlled paths with $\varepsilon=1$ for (b) $p=1$ and (c) $p=\infty$. 
Colored arrows show admissible $\varepsilon$-$\ell^p$-controlled paths from $x=(0,6)^{\rm T}$, and dashed arrows represent transitions induced by $\varepsilon$-control.
The choice of $p$ changes the set of allowable perturbations and therefore the reachable branches of the dynamics; the $\ell^\infty$ constraint permits additional connections that are not available under the $\ell^1$ constraint. Black dots indicate the element of $[x]_f^{\varepsilon\text{-}\ell^p}$.
}
\label{fig:exmple_path}
\end{figure*} 

By definitions of $[x]_f^{\varepsilon\text{-}\ell^p}$ and $x \overset{\exists}{\underset{f,\varepsilon\text{-}\ell^p}{\rightharpoonup}} A$, we have the following observation.

\begin{lemma}
The following statements are equivalent for any $\varepsilon \in [-\infty,\infty]$, any point $x \in X$, and any subset $A \subseteq X$:
\begin{enumerate}[\rm (1)]
\item $x \overset{\exists}{\underset{f,\varepsilon\text{-}\ell^p}{\rightharpoonup}} A$
\item $[x]_f^{\varepsilon\text{-}\ell^p} \cap A \neq \emptyset$. 
\end{enumerate}
\end{lemma}

\subsection{Reachable values and reachable points}

For sets $B, C$, the symbol $C - B$ is used instead of the set difference $C \setminus B$ when $B \subseteq C$.

From now on, in this section, let $f \colon X \rightharpoonup X$ be a partial map on a set $X$ with a cost function $c \colon X^2 \to [-\infty,\infty]$, $p \in [1,\infty]$ a value, and $h \colon X \rightharpoonup [-\infty, \infty]$ a partial map. 
In this paper, note that $h$ can be considered as an evaluation partial function.

\subsubsection{$(c,\varepsilon)$-fixed points}

We introduce the following concepts to express the property that points outside the domain of the mapping remain stationary even under perturbation. 

\begin{definition}
For any $\varepsilon \in [-\infty,\infty)$, a point $x \in X$
is a {\bf $\bm{(c,\varepsilon)}$-fixed point} if $c(x,\cdot)^{-1}([-\infty,\varepsilon]) \setminus \{ x \} = \emptyset$. 
\end{definition}

Denote by $E_{c}(\varepsilon)$ the set of $(c,\varepsilon)$-fixed points. 

\begin{definition}
A point $x \in X $
is a {\bf $\bm{c}$-fixed point} if $c(x,\cdot)^{-1}(\infty) = X - \{ x \}$. 
\end{definition}

Denote by $E_{c}$ the set of $c$-fixed points. 
We have the following statement. 

\begin{lemma}\label{lem:002}
The following statements are equivalent for any point $x \in X$ with $c(x,x) < \infty$: 
\begin{enumerate}[\rm (1)]
\item The point $x$ is a $c$-fixed point. 
\item The point $x$ is a $(c,\varepsilon)$-fixed point for any $\varepsilon \in [-\infty,\infty)$. 
\end{enumerate}
\end{lemma}

\begin{proof}
Fix any point $x \in X$ with $c(x,x) < \infty$. 
Suppose that $x$ is a $c$-fixed point. 
For any $\varepsilon \in [-\infty,\infty)$, 
we have that  $c(x,\cdot)^{-1}([-\infty,\varepsilon]) \subseteq \{ x \}$ and so that $c(x,\cdot)^{-1}([-\infty,\varepsilon]) \setminus \{ x \} = \emptyset$. 
Conversely, suppose that $x$ is a $(c,\varepsilon)$-fixed point for any $\varepsilon \in [-\infty,\infty)$. 
By $c(x,x) < \infty$, we have that $c(x,\cdot)^{-1}([-\infty,\infty)) - \{ x \} = \bigcup_{\varepsilon \in [-\infty,\infty)}c(x,\cdot)^{-1}([-\infty,\varepsilon]) \setminus \{ x \} = \emptyset$. 
This means that $c(x,\cdot)^{-1}(\infty) = X - \{ x \}$.
\end{proof}

The previous lemma implies the following statement. 

\begin{corollary}
Suppose that the cost function $c$ is non-degenerate. 
Then the following statements are equivalent for any point $x \in X$: 
\begin{enumerate}[\rm (1)]
\item The point $x$ is a $c$-fixed point. 
\item The point $x$ is a $(c,\varepsilon)$-fixed point for any $\varepsilon \in [-\infty,\infty)$. 
\end{enumerate}
\end{corollary}

Notice that $E_c = \bigcap_{\varepsilon \in [-\infty,\infty)} E_{c}(\varepsilon)$ when $c$ is non-degenerate.

\subsubsection{Reachable ranges, reachable domains, and sublevel controllability ranges}

For any $\varepsilon \in [-\infty,\infty]$ and any $x \in X$, define a subset $R_{h,f}^{\ell^p}(\varepsilon,x) \subseteq [-\infty, \infty]$, called {\bf $\bm{(\varepsilon,h)}$-reachable range from $\bm{x}$}, as follows: 
\[
R_{h,f}^{\ell^p}(\varepsilon,x) := \left\{ r \in [-\infty, \infty] \middle| x \overset{\exists}{\underset{f,\varepsilon\text{-}\ell^p}{\rightharpoonup}} h^{-1}(r) \right\} = h\left([x]^{\varepsilon\text{-}\ell^p}_f \cap \mathop{\mathrm{dom}}h\right) 
\]

For any $\varepsilon \in [-\infty,\infty]$ and any $r \in [-\infty, \infty]$, define a subset $D_{h,f}^{\ell^p}(\varepsilon,r) \subseteq X$, called {\bf $\bm{(\varepsilon,h)}$-reachable domain to $\bm{r}$}, as follows: 
\[
D_{h,f}^{\ell^p}(\varepsilon,r) := \left\{ x \in X \middle| x \overset{\exists}{\underset{f,\varepsilon\text{-}\ell^p}{\rightharpoonup}} h^{-1}(r) \right\}
\]
Then $D_{h,f}^{\ell^p}(\varepsilon) := \bigcup_{r \in [-\infty, \infty]} D_{h,f}^{\ell^p}(\varepsilon,r)$ is called {\bf $\bm{(\varepsilon,h)}$-reachable domain}.
Notice that $D_{h,f}^{\ell^p}(\varepsilon)$ is the set of points from which the domain of $h$ is reachable via an $\varepsilon$-$\ell^p$-controlled path.

For any $x \in X$ and any $r \in [-\infty, \infty]$, define a subset $C_{h,f}^{\ell^p}(x,r) \subseteq [-\infty, \infty]$, called {\bf $\bm{h}$-sublevel controllability range}, and $\underline{C}_{h,f}^{\ell^p} \colon X \times [-\infty, \infty] \to [0,\infty]$ as follows: 
\[
C_{h,f}^{\ell^p}(x,r) := \left\{ \varepsilon \in [-\infty,\infty] \middle| x \overset{\exists}{\underset{f,\varepsilon\text{-}\ell^p}{\rightharpoonup}} h^{-1}([-\infty,r]) \right\}
\]
\[
\underline{C}_{h,f}^{\ell^p}(x,r) :=
\begin{cases}
\infty & \text{ if } C_{h,f}^{\ell^p}(x,r) = \emptyset\\
\inf C_{h,f}^{\ell^p}(x,r) & \text{ otherwise}
\end{cases}
\]
Here, we consider that ``$\inf \emptyset = \infty$''. 
Therefore, put $\inf \emptyset := \infty$. 
Then $\underline{C}_{h,f}^{\ell^p}(x,r) =\inf C_{h,f}^{\ell^p}(x,r)$. 
Examples of these notions are illustrated in Fig.~\ref{fig:exmple_RDC}.
We have the following equivalence to the non-existence of $\varepsilon$-$\ell^p$-controlled paths from a certain initial point to the domain of $h$. 

\begin{figure*}[t]
\begin{center}
\makebox[\textwidth][c]{%
\includegraphics[width=1.\textwidth]{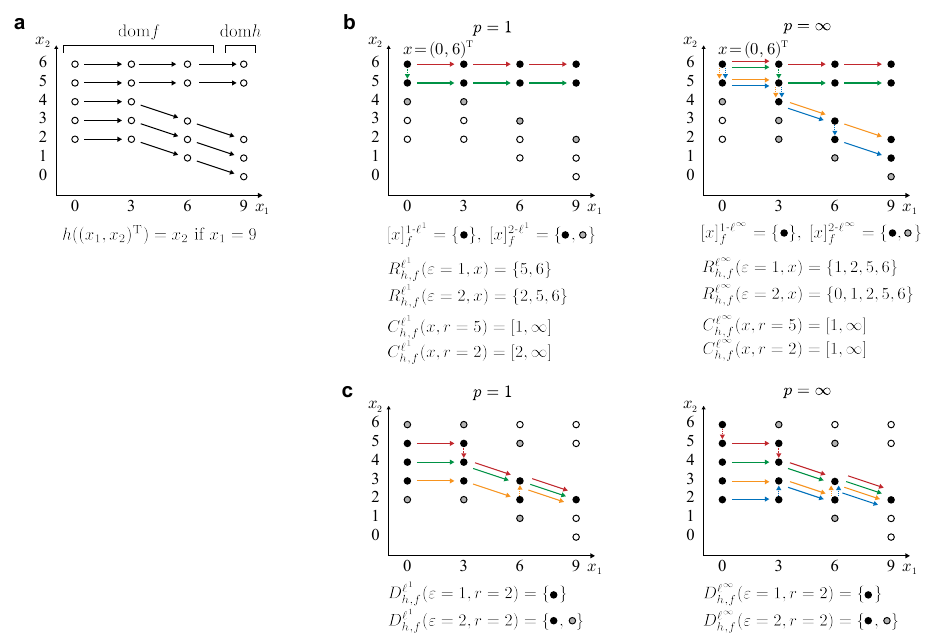}
}
\end{center} 
\caption{
Illustration of reachable sets and controllability ranges indexed by the attracting level.
\textbf{a}, A schematic setting of a dynamical system $f$ and a function $h$ on the two-dimensional state space. 
The setting of the set $X$ and the dynamical system $f$ follows that in Fig.~\ref{fig:exmple_path}. 
The function $h$ is defined on the states with $x_1=9$, and its value depends on the vertical coordinate.
\textbf{b}, Examples of $(\varepsilon, r)$-reachable ranges $R^{\ell^p}_{h, f}$ and the $h$-sublevel controllability range $C^{\ell^p}_{h, f}$ for $\varepsilon=1$ and $\varepsilon=2$. 
Colored paths represent admissible controlled trajectories from the initial state $x$ to the $h$-sublevel set $h^{-1}([-\infty,r])$. 
The left and right panels show the cases $p=1$ and $p=\infty$, respectively. 
The choice of $p$ changes the set of admissible perturbations and therefore changes the reachable branches and the controllability range.
\textbf{c}, Examples of the $(\varepsilon, h)$-reachable domain $D_{h,f}^{\ell^p}$ for $p=1$ and $p=\infty$. 
The colored paths indicate states that can reach the $h$-sublevel set under $\varepsilon$-$\ell^p$ control. 
Comparing the two panels shows how the norm used to measure the control affects the states that are regarded as controllable at each attracting level.
}
\label{fig:exmple_RDC}
\end{figure*} 

\begin{lemma}\label{lem:V_eq}
The following statements are equivalent for any $x \in X$ and any $\varepsilon \in [-\infty,\infty]$: 
\begin{enumerate}[\rm (1)]
\item $x \overset{\exists}{\underset{f,\varepsilon\text{-}\ell^p}{\not\rightharpoonup}} \mathop{\mathrm{dom}}h$. 
\item $[x]_f^{\varepsilon\text{-}\ell^p} \cap \mathop{\mathrm{dom}}h = \emptyset$.
\item $R_{h,f}^{\ell^p}(\varepsilon,x) = \emptyset$. 
\item $\varepsilon \notin C_{h,f}^{\ell^p}(x,\infty)$.
\item $x \notin D_{h,f}^{\ell^p}(\varepsilon)$. 
\end{enumerate}
\end{lemma}

\begin{proof}
By definitions of $[x]_f^{\varepsilon\text{-}\ell^p}$, assertions (1) and (2) are equivalent. 
From definitions of $R_{h,f}^{\ell^p}(\varepsilon,x)$, assertions (1) and (3) are equivalent. 
By definitions of $C_{h,f}^{\ell^p}(x,\infty)$, assertions (1) and (4) are equivalent. 
From definitions of $D_{h,f}^{\ell^p}(\varepsilon)$, assertions (1) and (5) are equivalent. 
\end{proof}

By the monotonicity of $\ell^p$-norms, the previous definitions imply the following statement. 

\begin{lemma}\label{lem:inclusion_V}
The following statement holds for any point $x \in X$, any $p \leq p' \in [1,\infty]$, any $\varepsilon \leq \varepsilon' \in [0,\infty]$, and any $r \leq r' \in [-\infty,\infty]$: 
\[
\begin{split}
R_{h,f}^{\ell^p}(\varepsilon,x) &\subseteq R_{h,f}^{\ell^p}(\varepsilon',x) \cap R_{h,f}^{\ell^{p'}}(\varepsilon,x)
\\
D_{h,f}^{\ell^p}(\varepsilon,r) &\subseteq D_{h,f}^{\ell^p}(\varepsilon',r) \cap D_{h,f}^{\ell^{p'}}(\varepsilon,r)
\\
C_{h,f}^{\ell^p}(x,r) &\subseteq C_{h,f}^{\ell^p}(x,r') \cap C_{h,f}^{\ell^{p'}}(x,r)
\\
\underline{C}_{h,f}^{\ell^p}(x,r) & \hspace{1pt} \geq \max \left\{ \underline{C}_{h,f}^{\ell^p}(x,r'), \underline{C}_{h,f}^{\ell^{p'}}(x,r) \right\}
\end{split}
\]
\end{lemma}

\section{A mapping extended from a partial map with respect to a dynamical system}

We introduce the following effect function. 



\begin{definition}
For any $\varepsilon \in [0,\infty]$, the extension $h_f^{\varepsilon\text{-}\ell^p} \colon X \rightharpoonup [-\infty, \infty]$ of $h$, called {\bf $\bm{\varepsilon}$-control effect function}, is defined for any point $x \in X$ as follows: 
\[
h_f^{\varepsilon\text{-}\ell^p} (x) := 
\begin{cases}
\infty & \text{if } x \overset{\exists}{\underset{f,\varepsilon\text{-}\ell^p}{\not\rightharpoonup}} \mathop{\mathrm{dom}}h \hspace{10pt} \left(\mathrm{i.e.} \hspace{5pt} R_{h,f}^{\ell^p}(\varepsilon,x) = \emptyset \right)\\
\inf R_{h,f}^{\ell^p}(\varepsilon,x) 
& \text{otherwise}
\end{cases}
\] 
\end{definition}

Since $\inf \emptyset = \infty$, we have that $h_f^{\varepsilon\text{-}\ell^p} (x) = \inf R_{h,f}^{\ell^p}(\varepsilon,x) = \inf_{} h\left([x]^{\varepsilon\text{-}\ell^p}_f \cap \mathop{\mathrm{dom}}h\right)$, because of definition of $R_{h,f}^{\ell^p}(\varepsilon,x)$. 
Roughly speaking, we can consider a small value of $h_f^{\varepsilon\text{-}\ell^p}$ to be a ``good'' cost, and a large value of $h_f^{\varepsilon\text{-}\ell^p}$ to be an ``expensive'' cost. 

For any $\varepsilon \in [0,\infty]$, we define the {\bf effect function} $\bm{h_f^{\Delta \varepsilon\text{-}\ell^p}(x)} := h_f^{\varepsilon\text{-}\ell^p}(x)- h_f^{0\text{-}\ell^p}(x) \leq 0$ {\bf of the difference}.  
Here, we adopt the convention that $\infty - \infty := 0$ and $r - \infty := -\infty$ for any $r \in [0,\infty]$.  

\begin{example}
In the setting of Fig.~\ref{fig:exmple_path}, the following hold: for $x=(0,6)^{\rm T}$, 
\begin{enumerate}[\rm (1)]
\item 
$
h_f^{0\text{-}\ell^1}(x)=6,\quad 
h_f^{1\text{-}\ell^1}(x)=5,\quad 
h_f^{2\text{-}\ell^1}(x)=2.
$
\item 
$
h_f^{0\text{-}\ell^\infty}(x)=6,\quad 
h_f^{1\text{-}\ell^\infty}(x)=1,\quad 
h_f^{2\text{-}\ell^\infty}(x)=0.
$
\item 
$
h_f^{\Delta 0\text{-}\ell^1}(x)=0,\quad 
h_f^{\Delta 1\text{-}\ell^1}(x)=-1,\quad 
h_f^{\Delta 2\text{-}\ell^1}(x)=-4.
$
\item 
$
h_f^{\Delta 0\text{-}\ell^\infty}(x)=0,\quad 
h_f^{\Delta 1\text{-}\ell^\infty}(x)=-5,\quad 
h_f^{\Delta 2\text{-}\ell^\infty}(x)=-6.
$
\end{enumerate}
\end{example}

We have the following observations.

\begin{lemma}\label{reduction_orbit}
Suppose that the cost function $c$ is non-negative-valued and non-degenerate. 
The following properties hold for any $x \in X$: 
\begin{enumerate}[\rm (1)]
\item $[x]_f^{0\text{-}\ell^p} = O^{\geq 0}_f(x)$.
\item If $x \in X - \mathop{\mathrm{dom}}f$, then $[x]_f^{0\text{-}\ell^p} = O^{\geq 0}_f(x) = \{ x \}$.
\end{enumerate}
\end{lemma}

\begin{lemma}\label{lem:correspond_extension_h}
If the cost function $c$ is non-negative-valued and non-degenerate and $\mathop{\mathrm{dom}}h \cap \mathop{\mathrm{dom}}f = \emptyset$, then $h_f^{0\text{-}\ell^p}\vert_{\mathop{\mathrm{dom}}h} = h$.
\end{lemma}

\begin{proof}
By Lemma~\ref{reduction_orbit}, the negative-valued property and the non-degeneracy of $c$ imply that $[x]_f^{0\text{-}\ell^p} = O^{\geq 0}_f(x) = \{ x \}$ for any $x \in X - \mathop{\mathrm{dom}}f$.
Fix any point $x \in \mathop{\mathrm{dom}}h$. 
By $\mathop{\mathrm{dom}}h \subseteq X - \mathop{\mathrm{dom}}f$, we have that $R_{h,f}^{\ell^p}(0,x) = \{ h(x) \}$ and so that $h_f^{0\text{-}\ell^p}(x) = \inf R_{h,f}^{\ell^p}(0,x) = h(x)$.
\end{proof}

\begin{lemma}\label{lem:non_deg_ext}
If the cost function $c$ is non-negative-valued and non-degenerate and $\mathop{\mathrm{dom}}h \cap \mathop{\mathrm{dom}}f = \emptyset$, then the intersection $[x]_f^{0\text{-}\ell^p} \cap \mathop{\mathrm{dom}}h$ for any $x \in X$ contains at most one element.
\end{lemma}

\begin{proof}
Fix any point $x \in X$. 
From $\mathop{\mathrm{dom}}h \cap \mathop{\mathrm{dom}}f = \emptyset$, the intersecion $\mathop{\mathrm{dom}}h \cap O^{\geq 0}_f(x)$  contains at most one element.
By the non-negative-valued property and the non-degeneracy of $c$, Lemma~\ref{reduction_orbit} implies that $[x]_f^{0\text{-}\ell^p} = O^{\geq 0}_f(x)$ and so that the intersecion $[x]_f^{0\text{-}\ell^p} \cap \mathop{\mathrm{dom}}h = \mathop{\mathrm{dom}}h \cap O^{\geq 0}_f(x)$  contains at most one element.
\end{proof}

Similarly, we have the following properties. 

\begin{lemma}\label{lem:exit}
For any $(c,\varepsilon)$-fixed point $x \in \mathop{\mathrm{dom}}h \setminus \mathop{\mathrm{dom}}f$, we have that $h_f^{\varepsilon\text{-}\ell^p}(x) = h(x)$.
In other words, we obtain that $h_f^{\varepsilon\text{-}\ell^p} = h$ on $(\mathop{\mathrm{dom}}h \cap E_{c}(\varepsilon)) \setminus \mathop{\mathrm{dom}}f$.
\end{lemma}

\begin{proof}
Fix any point $(c,\varepsilon)$-fixed point $x \in \mathop{\mathrm{dom}}h \setminus \mathop{\mathrm{dom}}f$. 
Then $[x]_f^{\varepsilon\text{-}\ell^p} = \{x\}$. 
This implies that $R_{h,f}^{\ell^p}(\varepsilon,x) = \{ h(x) \}$ and so that $h_f^{\varepsilon\text{-}\ell^p}(x) = \inf R_{h,f}^{\ell^p}(\varepsilon,x) = h(x)$.
\end{proof}

\begin{lemma}\label{lem:non_deg_ext_epsilon}
If $\mathop{\mathrm{dom}}h \subseteq E_{c}(\varepsilon) \setminus \mathop{\mathrm{dom}}f$,
then the intersection $N(\gamma) \cap \mathop{\mathrm{dom}}h$ for any $\varepsilon$-controlled path $\gamma$ in $X$ consists of at most one point, which is the terminal state of $\gamma$ if it exists.
\end{lemma}

\begin{proof}
For any point $x \in \mathop{\mathrm{dom}}h$, we obtain that $[x]_f^{\varepsilon\text{-}\ell^p} = \{x\}$. 
Therefore, any $\varepsilon$-controlled path in $X$ does not contain points of $\mathop{\mathrm{dom}}h$ except for its terminal state.
This implies the assertion.  
\end{proof}

\subsection{$-\varepsilon$-control}

\begin{definition}
For any $\varepsilon \in [0,\infty]$, the extension $h_f^{-\varepsilon\text{-}\ell^p} \colon X \to [-\infty, \infty]$ of $h$, called {\bf $\bm{-\varepsilon}$-control effect function}, is defined for any point $x \in X$ as follows: 
\[
h_f^{-\varepsilon\text{-}\ell^p} (x) := 
\begin{cases}
\infty & \text{if } x \overset{\exists}{\underset{f,0\text{-}\ell^p}{\not\rightharpoonup}} \mathop{\mathrm{dom}}h \hspace{10pt} \left(\mathrm{i.e.} \hspace{5pt} R_{h,f}^{\ell^p}(0,x) = \emptyset \right) \\
\sup R_{h,f}^{\ell^p}(\varepsilon,x) & \text{otherwise}
\end{cases}
\] 
\end{definition}


\begin{definition}
For any $\varepsilon \in [0,\infty]$, we define the {\bf effect function} $\bm{h_f^{\Delta -\varepsilon\text{-}\ell^p}(x)} := h_f^{-\varepsilon\text{-}\ell^p}(x)- h_f^{0\text{-}\ell^p}(x) \geq 0$ {\bf of the difference}.  
\end{definition}

\begin{definition}
Define the {\bf effect function} $\bm{h_f^{\Delta}} \colon [1,\infty] \times [-\infty,\infty] \times X \to [-\infty,\infty]$ {\bf of the difference} by $h_f^{\Delta}(p,\varepsilon, \cdot) := h_f^{\Delta \varepsilon \text{-}\ell^p}$.  
\end{definition}

We have the following observation. 

\begin{lemma}\label{lem:correspondence_zero}
Let $h \colon X \rightharpoonup [-\infty, \infty]$ be a 
partial map.
If $\mathop{\mathrm{dom}}h \cap \mathop{\mathrm{dom}} f  = \emptyset$ and the cost function $c$ is non-negative-valued and non-degenerate, then $h_f^{0\text{-}\ell^p} = h_f^{-0\text{-}\ell^p}$.
\end{lemma}

\begin{proof}
Suppose that $\mathop{\mathrm{dom}}h \cap \mathop{\mathrm{dom}} f  = \emptyset$ and the cost function $c$ is non-negative-valued and non-degenerate. 
Lemma~\ref{lem:non_deg_ext} implies that $\vert[x]_f^{0\text{-}\ell^p} \cap \mathop{\mathrm{dom}}h \vert \leq 1$ for any $x \in X$.
Fix any $x \in X$. 
If $[x]_f^{0\text{-}\ell^p} \cap \mathop{\mathrm{dom}}h = \emptyset$, then $R_{h,f}^{\ell^p}(0,x) = \emptyset$ and so $h_f^{0\text{-}\ell^p} (x) = \infty = h_f^{-0\text{-}\ell^p} (x)$, because of Lemma~\ref{lem:V_eq}.
Thus, we may assume that $[x]_f^{0\text{-}\ell^p} \cap \mathop{\mathrm{dom}}h$ consists of one point. 
Then $h_f^{0\text{-}\ell^p} (x) = \inf R_{h,f}^{\ell^p}(0,x) = \sup R_{h,f}^{\ell^p}(0,x) = h_f^{-0\text{-}\ell^p} (x)$. 
\end{proof}

By definition, we have the following inequalities. 

\begin{lemma}\label{lem:decreasing}
The following statement holds for any point $x \in X$ and any $\varepsilon \leq \varepsilon' \in (0,\infty]$: 
\[
h_f^{\varepsilon'\text{-}\ell^p}(x) \leq h_f^{\varepsilon\text{-}\ell^p}(x) \leq h_f^{0\text{-}\ell^p}(x) \leq h_f^{-0\text{-}\ell^p}(x) \leq h_f^{-\varepsilon\text{-}\ell^p}(x) \leq h_f^{-\varepsilon'\text{-}\ell^p}(x)
\]
\end{lemma}

\begin{proof}
Fix any $x \in X$ and any $\varepsilon \leq \varepsilon' \in [0,\infty]$. 
By definition of $R_{h,f}^{\ell^p}$, we have that $R_{h,f}^{\ell^p}(\varepsilon,x) \subseteq R_{h,f}^{\ell^p}(\varepsilon',x)$. 
From $h_f^{\varepsilon''\text{-}\ell^p} (x) = \inf R_{h,f}^{\ell^p}(\varepsilon'',x)$ for any $\varepsilon'' \in [0,\infty]$, we obtain that $h_f^{0\text{-}\ell^p}(x) \geq h_f^{\varepsilon\text{-}\ell^p}(x) \geq h_f^{\varepsilon'\text{-}\ell^p}(x)$. 
If $x \overset{\exists}{\underset{f,0\text{-}\ell^p}{\not\rightharpoonup}} \mathop{\mathrm{dom}}h$, then Lemma~\ref{lem:V_eq} implies that $h_f^{-\varepsilon'\text{-}\ell^p}(x) = \infty$ for any $\varepsilon' \in [0,\infty]$, which implies the assertion in this case. 
Thus, we may assume that $x \overset{\exists}{\underset{f,0\text{-}\ell^p}{\rightharpoonup}} \mathop{\mathrm{dom}}h$. 
By $R_{h,f}^{\ell^p}(\varepsilon,x) \subseteq R_{h,f}^{\ell^p}(\varepsilon',x)$, since $h_f^{-\varepsilon''\text{-}\ell^p} (x) = \sup R_{h,f}^{\ell^p}(\varepsilon'',x)$ for any $\varepsilon'' \in [0,\infty]$, we obtain $h_f^{0\text{-}\ell^p}(x) = \inf R_{h,f}^{\ell^p}(0,x) \leq \sup R_{h,f}^{\ell^p}(0,x) =h_f^{-0\text{-}\ell^p}(x) \leq h_f^{-\varepsilon\text{-}\ell^p}(x) \leq h_f^{-\varepsilon'\text{-}\ell^p}(x)$.
\end{proof}

\begin{corollary}
For any $x \in X$, the effect function $h_f^{\Delta \cdot \text{-}\ell^p}(x) \colon [-\infty,\infty] \to [-\infty,\infty]$ is weakly decreasing.
\end{corollary}

Moreover, we have the following relation. 

\begin{lemma}\label{lem:ineq_h}
The following statement holds for any point $x \in X$, any $\varepsilon \in [-\infty,\infty]$, and any $p \leq p' \in [1,\infty]$: 
\[
h_f^{\varepsilon\text{-}\ell^{p'}} (x) \leq h_f^{\varepsilon\text{-}\ell^p} (x) \leq h_f^{-\varepsilon\text{-}\ell^p} (x) \leq  h_f^{-\varepsilon\text{-}\ell^{p'}} (x) 
\]
\end{lemma}

\begin{proof}
By the monotonicity of $\ell^p$-norms, the inequality holds.
\end{proof}

\subsubsection{Control effect function}

\begin{definition}
The extension $h_f^{\ell^p} \colon ([-\infty,-0] \sqcup [0, \infty]) \times X \to [-\infty, \infty]$ of $h$, called {\bf control effect function with respect to $\bm{\ell^p}$}, is defined by $h_f^{\ell^p}(\varepsilon,x) := h_f^{\varepsilon\text{-}\ell^p} (x)$. 
\end{definition}

\begin{definition}
The extension $h_f \colon [1,\infty] \times ([-\infty,-0] \sqcup [0, \infty]) \times X \to [-\infty, \infty]$ of $h$, called {\bf control effect function}, is defined by $h_f(p,\varepsilon,x) := h_f^{\ell^p}(\varepsilon,x)$.
\end{definition}

\subsubsection{Properties of control effect functions}

We have the following statement.

\begin{proposition}\label{th:h_f}
Let $f \colon X \rightharpoonup X$ be a partial map on a set $X$ with a cost function $c \colon X^2 \to [-\infty,\infty]$. 
Then the extension $h_f \colon [1,\infty] \times ([-\infty,-0] \sqcup [0, \infty]) \times X \to [-\infty, \infty]$ of a partial map $h \colon X \rightharpoonup [-\infty, \infty]$ satisfies the following properties: 
\begin{enumerate}[\rm (1)]
\item The function $h_f$ is weakly decreasing in the second component.
\item For any $\varepsilon \in [-\infty,\infty]$ and any $p \leq p' \in [1,\infty]$, we have the following inequality: 
\[
h_f(p',\varepsilon, \cdot) \leq h_f(p,\varepsilon, \cdot) \leq h_f(p,-\varepsilon, \cdot) \leq  h_f(p',-\varepsilon, \cdot)
\]
\item If $\mathop{\mathrm{dom}}h \cap \mathop{\mathrm{dom}} f  = \emptyset$ and $c$ is non-negative-valued and non-degenerate, then $h_f(\cdot, -0, \cdot) = h_f(\cdot, 0, \cdot)$. 
\end{enumerate}
\end{proposition}

\begin{proof}
Lemma~\ref{lem:decreasing} implies that $h_f$ is weakly decreasing in the second component. 
By Lemma~\ref{lem:ineq_h}, assertion (2) holds. 
Lemma~\ref{lem:correspondence_zero} implies the equality $h_f(\cdot, -0, \cdot) = h_f(\cdot, 0, \cdot)$. 
\end{proof}





Properties of an $\varepsilon$-controlled path to the domain $\mathop{\mathrm{dom}}h$ are described as follows. 

\begin{proposition}\label{prop:e_controll}
Suppose that $X = \mathop{\mathrm{dom}} f  \sqcup \mathop{\mathrm{dom}}h = \sqcup_{k = 0}^\infty f^{-k}(\mathop{\mathrm{dom}}h)$. 
The following statements hold for any $\varepsilon \in [0,\infty)$, any $x \in X$, and any $\varepsilon$-conrolled path $\gamma$ from $x$ to a $(c,\varepsilon)$-fixed point $y \in \mathop{\mathrm{dom}}h \cap E_{c}(\varepsilon)$ with $h_f^{\varepsilon\text{-}\ell^p}(y) = h_f^{\varepsilon\text{-}\ell^p}(x)$: 
\begin{enumerate}[\rm (1)]
\item 
$h(y) = h_f^{\varepsilon\text{-}\ell^p}(y) = \min_{} h\left([x]^{\varepsilon\text{-}\ell^p}_f \cap \mathop{\mathrm{dom}}h\right)$.
\item For any point $x' \in N(\gamma)$, we have the following equality: 
\[
h_f^{\varepsilon\text{-}\ell^p}(x') = \min_{} h\left([x']^{\varepsilon\text{-}\ell^p}_f \cap \mathop{\mathrm{dom}}h\right) 
= h(y) 
\]
\end{enumerate}
\end{proposition}

\begin{proof}
Since $X = \sqcup_{n = 0}^\infty f^{-n}(\mathop{\mathrm{dom}}h)$, for any $x' \in X$, there is a non-negative integer $n \in \Z_{\geq 0}$ such that $f^n(x') \in \mathop{\mathrm{dom}}h$. 
From the definition of $h_f^{\varepsilon\text{-}\ell^p}$, we obtain that $h_f^{\varepsilon\text{-}\ell^p}(x') = \inf_{} h\left([x']^{\varepsilon\text{-}\ell^p}_f \cap \mathop{\mathrm{dom}}h\right)$ for any $x' \in X$. 

By the definition of $(c,\varepsilon)$-fixed point, we have that $[y]_f^{\varepsilon\text{-}\ell^p} = \{y\} = O^{\geq 0}_f(y)$.
From $\mathop{\mathrm{dom}} f \cap \mathop{\mathrm{dom}}h = \emptyset$, Lemma~\ref{lem:exit} implies that $h_f^{\varepsilon\text{-}\ell^p}(y) = h(y)$.
By the hypothesis $h_f^{\varepsilon\text{-}\ell^p}(y) = h_f^{\varepsilon\text{-}\ell^p}(x)$, we obtain $h(y) = h_f^{\varepsilon\text{-}\ell^p}(y) = h_f^{\varepsilon\text{-}\ell^p}(x) = \min_{} h\left([x]^{\varepsilon\text{-}\ell^p}_f \cap \mathop{\mathrm{dom}}h\right)$, which implies assertion (1).

Fix any point $x' \in N(\gamma)$. 
Since $y \in [x']^{\varepsilon\text{-}\ell^p}_f \subseteq [x]^{\varepsilon\text{-}\ell^p}_f$, we have the following equality: 
\[
\begin{split}
h(y) &\geq \min_{} h\left([x']^{\varepsilon\text{-}\ell^p}_f \cap \mathop{\mathrm{dom}}h\right) = h_f^{\varepsilon\text{-}\ell^p}(x')
\\
&\geq \min_{} h\left([x]^{\varepsilon\text{-}\ell^p}_f \cap \mathop{\mathrm{dom}}h\right)
= h_f^{\varepsilon\text{-}\ell^p}(x) = h_f^{\varepsilon\text{-}\ell^p}(y) = h(y)   
\end{split}
\]
\end{proof}

By the construction, if $\mathop{\mathrm{dom}}h \subseteq E_c \setminus \mathop{\mathrm{dom}}f$, 
then, for any $\varepsilon \in [0,\infty)$, every $\varepsilon$-controlled path from a point $x \in X$ to a point in the inverse image $h^{-1}(h_f^{\varepsilon\text{-}\ell^p}(x))$ minimizes the value of $h$ at its terminal point, where the minimum is taken over the terminal points of all $\varepsilon$-controlled paths from $x$.


\subsection{Existence of a ``minimizing'' $\varepsilon$-controlled path}

The existence of a ``minimizing'' $\varepsilon$-controlled path holds as follows. 

\begin{theorem}\label{th:minimizing}
Suppose that $\mathop{\mathrm{dom}}h \subseteq E_c(\varepsilon) \setminus \mathop{\mathrm{dom}}f$
for some $\varepsilon \in [0,\infty)$. 
For any point $x \in X$ with $h_f^{\varepsilon\text{-}\ell^p}(x) < \infty$ such that $[x]^{\varepsilon\text{-}\ell^p}_f \cap \mathop{\mathrm{dom}}h$ is a finite subset, there is an $\varepsilon$-controlled path $\gamma = (x;x_0, \ldots , x_{n-1};f(x_{n-1}))$ from $x$ satisfying the following properties: 
\begin{enumerate}[\rm (1)]
\item $h(f(x_{n-1})) = h_f^{\varepsilon\text{-}\ell^p}(x)$. 
\item $N(\gamma) \cap \mathop{\mathrm{dom}}h = \{ f(x_{n-1}) \}$. 
\item For any point $x' \in N(\gamma)$, we have the following equality: 
\[
h_f^{\varepsilon\text{-}\ell^p}(x') = \min_{} h\left([x']^{\varepsilon\text{-}\ell^p}_f \cap \mathop{\mathrm{dom}}h\right) 
= h(f(x_{n-1})) 
\]
\end{enumerate}
\end{theorem}

\begin{proof}
Since $h_f^{\varepsilon\text{-}\ell^p}(x) < \infty$, we obtain that $R_{h,f}^{\ell^p}(\varepsilon,x) \neq \emptyset$. 
By definition of $h_f^{\varepsilon\text{-}\ell^p}$, the finiteness of $[x]^{\varepsilon\text{-}\ell^p}_f\cap \mathop{\mathrm{dom}}h$ implies the following equality: 
\[
\begin{split}
h_f^{\varepsilon\text{-}\ell^p} (x) = \inf R_{h,f}^{\ell^p}(\varepsilon,x) &= \inf_{} h\left([x]^{\varepsilon\text{-}\ell^p}_f \cap \mathop{\mathrm{dom}}h\right)
\\
&= \min_{} h\left([x]^{\varepsilon\text{-}\ell^p}_f \cap \mathop{\mathrm{dom}}h\right) = \min R_{h,f}^{\ell^p}(\varepsilon,x)    
\end{split}
\]
Therefore, there is an $\varepsilon$-controlled path $\gamma = (x;x_0, \ldots , x_{n-1};f(x_{n-1}))$ from $x$ to $y := f(x_{n-1})$ such that $h(y) = h_f^{\varepsilon\text{-}\ell^p}(x)$, which implies assertion (1). 
Since $\mathop{\mathrm{dom}}h \subseteq E_c(\varepsilon) \setminus \mathop{\mathrm{dom}}f$, 
Lemma~\ref{lem:non_deg_ext_epsilon} implies $N(\gamma) \cap \mathop{\mathrm{dom}}h = \{ y \}$, which implies assertion (2).

Fix any point $x' \in N(\gamma)$. 
Then $y \in [x']^{\varepsilon\text{-}\ell^p}_f \subseteq [x]^{\varepsilon\text{-}\ell^p}_f$ and so $\min R_{h,f}^{\ell^p}(\varepsilon,x) = h_f^{\varepsilon\text{-}\ell^p}(x) = h(y) \in R_{h,f}^{\ell^p}(\varepsilon,x') \subseteq R_{h,f}^{\ell^p}(\varepsilon,x)$. 
Therefore, $h(y) = \min R_{h,f}^{\ell^p}(\varepsilon,x') = \min_{} h\left([x']^{\varepsilon\text{-}\ell^p}_f \cap \mathop{\mathrm{dom}}h\right) = h_f^{\varepsilon\text{-}\ell^p}(x')$, which implies assertion (3). 
\end{proof}

\section{Implementation of the definition for finite sets}

The finiteness implies the following algorithm.


\begin{theorem}\label{th:01}
Let $f \colon X \rightharpoonup X$ be a partial map on a finite set $X$ with a cost function $c \colon X^2 \to [-\infty,\infty]$, $p \in [1,\infty]$ a value, and $h \colon X \rightharpoonup [-\infty, \infty]$ a partial map. 
Suppose that $X = \mathop{\mathrm{dom}} f  \sqcup \mathop{\mathrm{dom}}h = \sqcup_{n = 0}^\infty f^{-n}(\mathop{\mathrm{dom}}h)$.
Put $r_0 := \min (\mathop{\mathrm{Im}} h)$.
For any $\varepsilon \in [0, \infty]$ with $\mathop{\mathrm{dom}}h \subseteq E_{c}(\varepsilon) \setminus \mathop{\mathrm{dom}} f$,
the function $h_f^{\varepsilon\text{-}\ell^p}$ coincides with the function $h_{f,\varepsilon\text{-}\ell^p} \colon X \to [-\infty, \infty]$ defined inductively as follows: 
\begin{enumerate}[\rm (1)]
\item $h_{f,\varepsilon\text{-}\ell^p}\vert_{\mathop{\mathrm{dom}}h} = h$. 
\item For any point $x \in X - \mathop{\mathrm{dom}} h$ with $x \overset{\exists}{\underset{f,\varepsilon\text{-}\ell^p}{\rightharpoonup}} h^{-1}(r_0)$, define the value $h_{f,\varepsilon\text{-}\ell^p}(x)$ as the level $r_0$. 
\item For any $r \in (r_0,\infty]$ and any point $x \in X - (\mathop{\mathrm{dom}} h \cup (h_{f,\varepsilon\text{-}\ell^p}^{-1}([-\infty,r))$ with $x \overset{\exists}{\underset{f,\varepsilon\text{-}\ell^p}{\rightharpoonup}} h^{-1}(r)$ {\rm (i.e.} $x \in D_{h,f}^{\ell^p}(\varepsilon,r) \setminus (\mathop{\mathrm{dom}} h \cup (h_{f,\varepsilon\text{-}\ell^p}^{-1}([-\infty,r))${\rm)}, define the value $h_{f,\varepsilon\text{-}\ell^p}(x)$ as the level $r$. 
\end{enumerate}
\end{theorem}

\begin{proof}
Fix any $x \in X$. 
Suppose that $x \in \mathop{\mathrm{dom}} h$.
Lemma~\ref{lem:exit} implies that $h_{f,\varepsilon\text{-}\ell^p}(x) = h(x) = h_f^{\varepsilon\text{-}\ell^p}(x)$.

Suppose that $x \in X - \mathop{\mathrm{dom}} h$ with $x \overset{\exists}{\underset{f,\varepsilon\text{-}\ell^p}{\rightharpoonup}} h^{-1}(r_0)$. 
Then  there is an $\varepsilon$-$\ell^p$-controlled path from $x$ to a terminal state $y \in h^{-1}(r_0) = h^{-1}(\min (\mathop{\mathrm{Im}} h))$.
By $y \in [x]_f^{\varepsilon\text{-}\ell^p}$ and $h(y) = \min (\mathop{\mathrm{Im}} h)$, we have that $h(y) = \min R_{h,f}^{\ell^p}(\varepsilon,x)$ and so that $h_{f,\varepsilon\text{-}\ell^p}(x) = r_0 = h(y) = \min R_{h,f}^{\ell^p}(\varepsilon,x) = h_f^{\varepsilon\text{-}\ell^p}(x)$.

Thus we may assume that $x \in X - (\mathop{\mathrm{dom}} h \cup (h_{f,\varepsilon\text{-}\ell^p}^{-1}([-\infty,r))$ with $x \overset{\exists}{\underset{f,\varepsilon\text{-}\ell^p}{\rightharpoonup}} h^{-1}(r)$ for some $r \in (r_0,\infty]$. 
Then 
\[
\min R_{h,f}^{\ell^p}(\varepsilon,x) = \min \left\{ r' \in [-\infty, \infty] \middle| x \overset{\exists}{\underset{f,\varepsilon\text{-}\ell^p}{\rightharpoonup}} h^{-1}(r') \right\} = r
\]
and so $h_{f,\varepsilon\text{-}\ell^p}(x) = r = \min R_{h,f}^{\ell^p}(\varepsilon,x) = h_f^{\varepsilon\text{-}\ell^p}(x)$. 
\end{proof}

By the definitions of $h_f^{\varepsilon\text{-}\ell^p}$ and $h_f^{-\varepsilon\text{-}\ell^p}$ under the nonempty condition $R_{h,f}^{\ell^p}(\varepsilon,y) = \emptyset$ for any $\varepsilon \in [-\infty,\infty]$ and any $y \in X$, the previous theorem implies the following statement. 

\begin{theorem}\label{th:02}
Let $f \colon X \rightharpoonup X$ be a partial map on a finite set $X$ with a cost function $c \colon X^2 \to [-\infty,\infty]$, $p \in [1,\infty]$ a value, and $h \colon X \rightharpoonup [-\infty, \infty]$ a partial map. 
Suppose that $X = \mathop{\mathrm{dom}} f  \sqcup \mathop{\mathrm{dom}}h = \sqcup_{n = 0}^\infty f^{-n}(\mathop{\mathrm{dom}}h)$.
Put $R_0 := \max (\mathop{\mathrm{Im}} h)$.
For any $\varepsilon \in [0, \infty]$ with $\mathop{\mathrm{dom}}h \subseteq E_{c}(\varepsilon) \setminus \mathop{\mathrm{dom}} f$,
the function $h_f^{-\varepsilon\text{-}\ell^p}$ coincides with the function $h_{f,-\varepsilon\text{-}\ell^p} \colon X \to [-\infty, \infty]$ defined inductively as follows: 
\begin{enumerate}[\rm (1)]
\item $h_{f,-\varepsilon\text{-}\ell^p}\vert_{\mathop{\mathrm{dom}}h} = h$. 
\item For any point $x \in X - \mathop{\mathrm{dom}} h$ with $x \overset{\exists}{\underset{f,\varepsilon\text{-}\ell^p}{\rightharpoonup}} h^{-1}(R_0)$ {\rm (i.e.} $x \in D_{h,f}^{\ell^p}(\varepsilon,R_0) \setminus \mathop{\mathrm{dom}} h${\rm)}, define the value $h_{f,-\varepsilon\text{-}\ell^p}(x)$ as the level $R_0$. 
\item For any $r \in [-\infty,R_0)$ and any point $x \in X - (\mathop{\mathrm{dom}} h \cup (h_{f,\varepsilon\text{-}\ell^p}^{-1}(r,\infty]))$ with $x \overset{\exists}{\underset{f,\varepsilon\text{-}\ell^p}{\rightharpoonup}} h^{-1}(r)$ {\rm(i.e.} $x \in D_{h,f}^{\ell^p}(\varepsilon,r) \setminus (\mathop{\mathrm{dom}} h \cup (h_{f,\varepsilon\text{-}\ell^p}^{-1}(r,\infty]))$ {\rm )}, define the value $h_{f,-\varepsilon\text{-}\ell^p}(x)$ as the level $r$. 
\end{enumerate}
\end{theorem}

\begin{proof}
Fix any $x \in X$. 
Suppose that $x \in \mathop{\mathrm{dom}} h$.
Then $[x]_f^{\varepsilon\text{-}\ell^p} = \{x\}$. 
This implies that $R_{h,f}^{\ell^p}(\varepsilon,x) = \{ h(x) \}$ and so that $h_f^{-\varepsilon\text{-}\ell^p}(x) = \sup R_{h,f}^{\ell^p}(\varepsilon,x) = h(x) = h_{f,-\varepsilon\text{-}\ell^p}(x)$.

Suppose that $x \in X - \mathop{\mathrm{dom}} h$ with $x \overset{\exists}{\underset{f,\varepsilon\text{-}\ell^p}{\rightharpoonup}} h^{-1}(R_0)$. 
Then  there is an $\varepsilon$-$\ell^p$-controlled path from $x$ to a terminal state $y \in h^{-1}(R_0) = h^{-1}(\max (\mathop{\mathrm{Im}} h))$.
By $y \in [x]_f^{\varepsilon\text{-}\ell^p}$ and $h(y) = \max (\mathop{\mathrm{Im}} h)$, we have that $h(y) = \max R_{h,f}^{\ell^p}(\varepsilon,x)$ and so that  $h_{f,-\varepsilon\text{-}\ell^p}(x) = R_0 = h(y) = \max R_{h,f}^{\ell^p}(\varepsilon,x) = h_f^{-\varepsilon\text{-}\ell^p}(x)$.

Thus we may assume that $x \in X - (\mathop{\mathrm{dom}} h \cup (h_{f,\varepsilon\text{-}\ell^p}^{-1}(r,\infty]))$with $x \overset{\exists}{\underset{f,\varepsilon\text{-}\ell^p}{\rightharpoonup}} h^{-1}(r)$ for some $r \in [-\infty,R_0)$. 
Then 
\[
\max R_{h,f}^{\ell^p}(\varepsilon,x) = \max \left\{ r' \in [-\infty, \infty] \middle| x \overset{\exists}{\underset{f,\varepsilon\text{-}\ell^p}{\rightharpoonup}} h^{-1}(r') \right\} = r
\]
and so $h_{f,\varepsilon\text{-}\ell^p}(x) = r = \max R_{h,f}^{\ell^p}(\varepsilon,x) = h_f^{-\varepsilon\text{-}\ell^p}(x)$. 
\end{proof}


\subsection{Multi-parameter-persistence}

For any tuple $(p,\varepsilon,\delta) \in [1,\infty] \times ([-\infty,-0] \sqcup [0,\infty]) \times [-\infty,\infty]$, we define a subset $H_{f,\varepsilon,\delta}^{\ell^p}$ as follows: 
\[
H_{f,\varepsilon,\delta}^{\ell^p} := \{ x \in X \mid h_f(p,\varepsilon,x) \leq \delta \} \subseteq X
\]

Define a set-valued function $H_f \colon [1,\infty] \times ([-\infty,-0] \sqcup [0,\infty]) \times [-\infty,\infty] \to 2^X$ by $H_f(p,\varepsilon,\delta) := H_{f,\varepsilon,\delta}^{\ell^p}$. 
We have the following multi-parameter filtrations. 

\begin{theorem}\label{th:03}
Let $f \colon X \rightharpoonup X$ be a partial map on a set $X$ with a cost function $c \colon X^2 \to [-\infty,\infty]$, and let $h \colon X \rightharpoonup [-\infty, \infty]$ be a partial map. 
The following statements hold:
\begin{enumerate}[\rm (1)]
\item For any $p \in [1,\infty]$, the family $(H_{f,\varepsilon,\delta}^{\ell^p})_{(\varepsilon,\delta) \in ([-\infty,-0] \sqcup [0,\infty]) \times [-\infty,\infty]}$ is a filtration.
\item The family $(H_{f,\varepsilon,\delta}^{\ell^p})_{(p,\varepsilon,\delta) \in [1,\infty] \times [0,\infty] \times [-\infty,\infty]}$ is a filtration. 
\end{enumerate}
\end{theorem}

\begin{proof}
By definition of $H_{f,\varepsilon,\delta}^{\ell^p}$, the function $H_{f,\varepsilon,\delta}^{\ell^p}$ is weakly increasing in the third component. 
From Proposition~\ref{th:h_f}, 
the function $h_f$ is weakly decreasing in the second component, and so the function $H_{f}(p,\cdot, \cdot)$ is weakly increasing in the second component with respect to the inclusion order, which implies assertion (1). 

From Lemma~\ref{lem:ineq_h}, the restriction $h_f\vert_{[1,\infty] \times [0,\infty] \times [-\infty,\infty]}$ is weakly decreasing in the first component, and so the function $H_{f}\vert_{[1,\infty] \times [0,\infty] \times [-\infty,\infty]}$ is weakly increasing in the second component with respect to the inclusion order, which implies assertion (2). 
\end{proof}

\subsubsection{Multi-parameter-persistence for control effect functions}

For any $(p,\varepsilon,\delta) \in [1,\infty] \times ([-\infty,-0] \sqcup [0,\infty]) \times [-\infty,\infty]$, we define a subset $H_{f,\Delta \varepsilon,\delta}^{\ell^p}$ as follows: 
\[
H_{f,\Delta \varepsilon,\delta}^{\ell^p} := \{ x \in X \mid h_f^{\Delta}(p,\varepsilon,x) \leq \delta \} \subseteq X
\]

Define a set-valued function $H_{f}^{\Delta} \colon [1,\infty] \times ([-\infty,-0] \sqcup [0,\infty]) \times [-\infty,\infty] \to 2^X$ by $H_f^\Delta(p,\varepsilon,\delta) := H_{f,\Delta \varepsilon,\delta}^{\ell^p}$. 

Since $H_f$ and $H_{f}^{\Delta}$ have the same ``piecewise'' monotonicity properties with respect to the first component $p$ and the second component $\varepsilon$, the following follows from the same argument in the proof of Theorem~\ref{th:03}.

\begin{theorem}\label{th:04}
Let $f \colon X \rightharpoonup X$ be a partial map on a set $X$ with a cost function $c \colon X^2 \to [-\infty,\infty]$, and let $h \colon X \rightharpoonup [-\infty, \infty]$ be a partial map. 
The following statements hold:
\begin{enumerate}[\rm (1)]
\item For any $p \in [1,\infty]$, the family $(H_{f,\Delta \varepsilon,\delta}^{\ell^p})_{(\varepsilon,\delta) \in ([-\infty,-0] \sqcup [0,\infty]) \times [-\infty,\infty]}$ is a filtration.
\item The family $(H_{f,\Delta \varepsilon,\delta}^{\ell^p})_{(p,\varepsilon,\delta) \in [1,\infty] \times [0,\infty] \times [-\infty,\infty]}$ is a filtration. 
\end{enumerate}
\end{theorem}

\subsection*{Acknowledgments}
This research was partially supported by the JST Moonshot R\&D Program and JST CREST (Grant Numbers JPMJMS2389 and JPMJCR24Q1).
The authors thank Dr. Enhao Liu for his valuable comments on topological data analysis.
T.Y. acknowledges the hospitality of the Fields Institute for Research in Mathematical Sciences.

\bibliographystyle{amsplain}
\bibliography{MPP}

\end{document}